\documentclass[12pt]{amsart}

\usepackage{amssymb, amscd, txfonts}
\usepackage{graphicx}

\numberwithin{equation}{section}

\newtheorem{theorem}{Theorem}[section]
\newtheorem{proposition}[theorem]{Proposition}
\newtheorem{lemma}[theorem]{Lemma}
\newtheorem{corollary}[theorem]{Corollary}

\theoremstyle{definition}

\newtheorem{example}[theorem]{Example}

\theoremstyle{remark}
\newtheorem{remark}[theorem]{Remark}

\newcommand{\Z}{\mathbb{Z}}
\newcommand{\Q}{\mathbb{Q}}
\newcommand{\R}{\mathbb{R}}
\newcommand{\C}{\mathbb{C}}
\newcommand{\HH}{\mathbb{H}}
\newcommand{\proj}{{\mathbb P}}
\newcommand{\QZ}{\mathbb{Q}/\mathbb{Z}}
\newcommand{\e}{{\mathbf e}}
\newcommand{\emu}{{\mathbf e}_{\mu}}
\newcommand{\elambda}{{\mathbf e}_{\lambda}}

\newcommand{\SL}{{\rm SL}_2(\mathbb{Z})}
\newcommand{\Mp}{{\rm Mp}_2(\mathbb{Z})}
\newcommand{\MGd}{M\Gamma_{0}(d)}
\newcommand{\Or}{{\rm O}^+}
\newcommand{\GL}{\Gamma_{L}}
\newcommand{\rk}{{\rm rk}}
\newcommand{\DL}{\mathcal{D}_{L}}
\newcommand{\DM}{\mathcal{D}_{M}}
\newcommand{\OAL}{{\rm O}(A_{L})}
\newcommand{\pull}{\uparrow_{L}^{L'}}
\newcommand{\push}{\downarrow_{L}^{L'}}
\newcommand{\pullA}{\uparrow_{I}^{A'}}
\newcommand{\pushA}{\downarrow_{I}^{A'}}
\newcommand{\ind}{{\rm ind}_{L}}

\newcommand{\ThetaK}{\Theta_{K}}

\newcommand{\divi}{{\rm div}}

\begin{document}

\title[]{Quasi-pullback of Borcherds products}
\author[]{Shouhei Ma}
\thanks{Supported by JSPS KAKENHI 15H05738 and 17K14158.} 
\address{Department~of~Mathematics, Tokyo~Institute~of~Technology, Tokyo 152-8551, Japan}
\email{ma@math.titech.ac.jp}
\maketitle

\begin{abstract}
Quasi-pullback of Borcherds products is an operation of renormalized restriction. 
It produces a meromorphic modular form on 
a lower dimensional symmetric domain which is again a Borcherds product. 
We give an explicit formula for the weakly holomorphic modular form of Weil representation type  
whose Borcherds lift is the quasi-pullback of the given Borcherds product. 
\end{abstract}


\section{Introduction}

Let $L$ be an even lattice of signature $(2, b)$. 
In \cite{Bo95}, \cite{Bo98}, Borcherds discovered a method for constructing 
meromorphic modular forms on the symmetric domain ${\DL}$ attached to $L$ 
whose divisor is a linear combination of Heegner divisors. 
His construction lifts weakly holomorphic modular forms $f$ of one variable 
with values in the Weil representation $\rho_{L}$ of $L$,  
and the principal part of $f$ determines 
the divisor and the weight of the resulting modular form $\Psi_{L}(f)$ on ${\DL}$. 
This orthogonal modular form $\Psi_{L}(f)$ is called the \textit{Borcherds product} associated to $f$. 

In some applications of Borcherds products, an operation called \textit{quasi-pullback}, 
first introduced by Borcherds in \cite{Bo95}, \cite{B-K-P-SB}, has played an important role. 
Let $M$ be a primitive sublattice of $L$ of signature $(2, b')$. 
If $\Psi$ is a Borcherds product on ${\DL}$, its quasi-pullback to ${\DM}$ is defined by 
first dividing $\Psi$ by zeros and poles containing ${\DM}$, 
and then restricting the resulting form to ${\DM}$. 
This produces a modular form on ${\DM}$ whose divisor and weight can be 
determined from those of $\Psi$ and the arithmetic information of the embedding $M\subset L$. 
If $\mathcal{D}_{M}$ is contained neither in the zero divisor nor in the pole divisor of $\Psi$, 
this is ordinary restriction. 
Quasi-pullback constructions have been applied to various problems, such as 
\begin{itemize}
\item the height formula for the Weyl vectors of Borcherds products (\cite{Bo95}), 
\item Borcherds lift for anisotropic lattices (\cite{Bo98}, see also \cite{H-MP}), 
\item the quasi-affineness of the moduli spaces of $K3$ surfaces (\cite{B-K-P-SB}), 
\item the Kodaira dimension of modular varieties (\cite{Ko99}, \cite{G-H-S07}, \cite{G-H-S13} et al),  
\item the analytic torsion of $K3$ surfaces with involutions (\cite{Yo98}, \cite{Yo13}), and  
\item generalized Kac-Moody algebras (\cite{G-N16}).  
\end{itemize}
In many of these examples, $L$ is the even unimodular lattice $II_{2,26}$ of signature $(2, 26)$ 
and $\Psi$ is the Borcherds form $\Phi_{12}$ constructed from $f=1/\Delta$ (\cite{Bo95}). 

Quasi-pullback of a Borcherds product is again a Borcherds product, 
at least when ${\rm rk}(M)\geq5$. 
In many cases this follows from Bruinier's converse theorem \cite{Br}, \cite{Br14}, 
and we show that this is always the case 
(provided that the Koecher principle holds for $\mathcal{D}_{M}$).  
Our main result is an explicit formula for the modular form of type $\rho_{M}$ 
whose Borcherds lift is the quasi-pullback of the given Borcherds product. 
Recall that the \textit{Witt index} of $M$ is 
the maximal rank of an isotropic sublattice of $M$. 

\begin{theorem}\label{main}
Let $\Psi_{L}(f)$ be the Borcherds product on ${\DL}$ associated to 
a weakly holomorphic modular form $f$ of type $\rho_{L}$. 
Let $K(-1)$ be a primitive negative-definite sublattice of $L$, with $K$ positive-definite.  
Assume that the Witt index of $M=K(-1)^{\perp}\cap L$ is smaller than ${\rm rk}(M)-2$. 
Then, up to a constant, 
the quasi-pullback of $\Psi_{L}(f)$ to ${\DM}$ is the Borcherds lift of 
the weakly holomorphic modular form of type $\rho_{M}$ defined by 
\begin{equation*}\label{eqn:main intro}
g = \langle f{\pull}, \: \Theta_{K} \rangle. 
\end{equation*}  
Here $L'=M\oplus K(-1)$, 
$\uparrow_{L}^{L'}$ is the pullback operation defined in Equation \eqref{eqn:pull and push}, 
$\Theta_{K}$ is the $\rho_{K}$-valued theta series of $K$, 
and $\langle \cdot, \Theta_{K} \rangle$ is the $\Theta$-contraction 
defined in Equation \eqref{eqn:Theta contraction}. 
\end{theorem}

Since $M$ has signature $(2, \ast)$, 
the Witt index condition is always satisfied when ${\rm rk}(M)\geq 5$. 
We need this condition only for the Koecher principle to hold.

In the case where $L$ is unimodular, so that $f$ is scalar-valued, 
the $\rho_M$-valued form $g$ becomes the product 
\begin{equation}\label{eqn:unimo formula}
g = f\cdot\Theta_{K}
\end{equation} 
under the isomorphism $\rho_M\simeq \rho_K$ 
(Example \ref{ex3}). 
Theorem \ref{main} in this version, 
especially for $(L, f)=(II_{2,26}, 1/\Delta)$, 
has been known to the experts, as can be found in the literature: 
\begin{itemize}
\item The first example is due to Borcherds (\S 16 of \cite{Bo95}) 
where $(L, f)=(II_{2,26}, 1/\Delta)$ and $M=II_{2,10}, II_{2,18}$.  
\item A sign can also be found in Theorem 13.1 of \cite{Bo95}  
where $(L, f)$ is general and $M=U\oplus\langle 2d \rangle$. 
Here $U$ is the even unimodular lattice of signature $(1, 1)$ and 
$\langle 2d \rangle$ is the rank $1$ lattice whose generators have norm $2d$.   
\item Another example appears in Theorem 8.5 of \cite{Yo98} 
where $(L, f)=(II_{2,26}, 1/\Delta)$ and $K$ is the Barnes-Wall lattice. 
\item A similar description can also be found in Remark 1 in \S 6 of \cite{G-H-S07}. 
\end{itemize}
We show that a similar formula holds more generally, 
with $f\cdot\Theta_{K}$ replaced by the tensor product $f\otimes\Theta_{K}$ (Example \ref{ex2}).  

In some applications, the $\rho_L$-valued form $f$ is constructed from 
a scalar-valued modular form $\varphi$ by means of ``induction'' 
(see, e.g., \cite{Bo00}, \cite{Scheit06}, \cite{Scheit09}, \cite{Scheit15}, \cite{Yo09}, \cite{Yo13}). 
In that case, the $\rho_M$-valued form $g$ can be expressed more explicitly 
in terms of $\varphi$ (\S \ref{ssec:theta contraction induction}). 
In a typical case, $g$ equals the induction from the scalar-valued form $\varphi \cdot \theta_{K}$, 
where $\theta_{K}$ is the scalar-valued theta series of $K$ 
(Corollary \ref{cor:induction and theta contraction split case}).

When $L$ contains $U \oplus U$, Gritsenko described the Borcherds lift 
in terms of the weak Jacobi forms of weight $0$ 
corresponding to the $\rho_L$-valued forms (see \S 3 of \cite{Gr12}). 
He proves that in the Jacobi form setting, 
quasi-pullback is given by the ordinary restriction of the source Jacobi form 
(see pp.16, 21, 23 of \cite{Gr10} for some examples). 
Theorem \ref{main}, in the case where $M$ contains $U \oplus U$, should be equivalent to 
(and gives a $\rho_L$-version proof of) Gritsenko's quasi-pullback formula. 
In fact, Theorem \ref{main} could be viewed as a unified generalization of 
the unimodular formula (Equation \eqref{eqn:unimo formula}) and the Gritsenko formula.

Theorem \ref{main} is proved by comparing the weights and the divisors of the two modular forms on ${\DM}$. 
Since we rely on the Koecher principle 
(and the fact that the character has finite order), 
the argument does not extend to the remaining case in ${\rm rk}(M)=3, 4$. 
But it seems plausible that the same formula would also hold in that case. 
At least we know that the two modular forms have the same weight and divisor. 

Schofer \cite{Schofer} considered an operation similar to the $\Theta$-contraction, 
at the level of the Schwarz space of $L\otimes \mathbb{A}_{f}$ 
(where $\mathbb{A}_{f}$ denotes the finite adeles) 
and for ${\rm rk}(M)=2$, 
to study CM values of Borcherds products. 
It may be the case that our $\Theta$-contraction is a finite version of Schofer's operation 
with general ${\rm rk}(M)$. 

Quasi-pullback of general holomorphic modular forms to rational quadratic divisors (i.e., ${\rm rk}(K)=1$) 
is systematically studied in \S 8.4 of \cite{G-H-S13}. 
The classical case $b=3, 2$, namely quasi-pullback 
from Siegel modular $3$-folds to Hilbert modular surfaces and 
from Hilbert modular surfaces to modular curves, 
has been also considered in \cite{Ao06}, \S 9 of \cite{vdG}, and \cite{Ao12}.

We thank K.~Yoshikawa for valuable remarks and for referring us to the paper \cite{Schofer},  
and V.~Gritsenko for kindly teaching us his Jacobi quasi-pullback formula. 
We also thank the referees for  
many detailed comments which helped us to improve the presentation.


\section{Weil representations and Borcherds products}\label{sec:preliminaries}

In this section we recall basic facts concerning 
Weil representations and Borcherds products (\cite{Bo98}, \cite{Br}). 
Let $L$ be an even lattice. 
By this we mean a free ${\Z}$-module of finite rank equipped with 
a symmetric bilinear form $(\: , \: ): L\times L\to {\Z}$ 
such that $(l, l)\in2{\Z}$ for every $l\in L$. 
The \textit{dual lattice} $L^{\vee}$ of $L$ is defined as the subgroup of $L_{{\Q}}$ 
consisting of vectors $l$ such that $(l, m)\in {\Z}$ for all $m\in L$. 
We write $q(l)=(l, l)/2$ for $l\in L^{\vee}$. 
Since $(l, m)\in {\Z}$ for every $l, m\in L$, we have $L\subset L^{\vee}$. 
The quotient $A_L=L^{\vee}/L$ is called the \textit{discriminant group} of $L$. 
Its natural ${\QZ}$-valued quadratic form $q: A_L\to{\QZ}$ 
is called the \textit{discriminant form} of $L$. 
The associated bilinear form 
\begin{equation*}
A_L \times A_L \to {\QZ}, \quad (\lambda, \mu)=q(\lambda+\mu)-q(\lambda)-q(\mu), 
\end{equation*}
is the reduction of the bilinear form on $L^{\vee}$ modulo ${\Z}$. 

In general, a finite abelian group $A$ equipped with 
a nondegenerate quadratic form $q:A\to{\QZ}$ is called a \textit{finite quadratic module}. 
We often abbreviate $(A, q)$ as $A$.  
Every finite quadratic module arises as the discriminant form of an even lattice (\cite{Ni}). 
We set $\sigma(A)=[b_{+}-b_{-}]\in{\Z}/8{\Z}$ 
where $A=A_{L}$ for an even lattice $L$ of signature $(b_{+}, b_{-})$. 
This value $\sigma(A)$ in ${\Z}/8{\Z}$ does not depend on the choice of $L$ 
such that $A=A_{L}$. 
The \textit{level} of $A$ is the smallest natural number $d$ such that 
$dq(\lambda)=0\in{\QZ}$ for all $\lambda\in A$.

\subsection{The Weil representation}\label{ssec:Weil rep}

Let ${\Mp}$ be the metaplectic double cover of ${\SL}$. 
It consists of elements of the form $(M, \phi)$ where 
$M=\begin{pmatrix}a & b \\ c & d \end{pmatrix}$ is an element of  ${\SL}$ 
and $\phi$ is a holomorphic function on the upper half plane ${\HH}$ such that $\phi(\tau)^2=c\tau+d$. 
It is known that ${\Mp}$ is generated by the two elements
\begin{equation*}
T = \left( \begin{pmatrix}1&1\\ 0&1\end{pmatrix}, 1 \right), \quad  
S = \left( \begin{pmatrix}0&-1\\ 1&0\end{pmatrix}, \sqrt{\tau} \right). 
\end{equation*}

Let $(A, q)$ be a finite quadratic module and let ${\C}A$ be the group algebra over $A$. 
For $\lambda\in A$ the corresponding basis vector of ${\C}A$ is denoted by ${\elambda}$. 
The \textit{Weil representation} $\rho_A$ of ${\Mp}$ 
is the unitary representation on ${\C}A$ defined by 
\begin{eqnarray*}
\rho_A(T)({\elambda}) & = & e(q(\lambda)){\elambda}, \\ 
\rho_A(S)({\elambda}) & = & 
\frac{e(-\sigma(A)/8)}{\sqrt{|A|}} \sum_{\mu\in A}e(-(\lambda, \mu)){\emu}, 
\end{eqnarray*}
where $e(z)={\rm exp}(2\pi i z)$ for $z\in{\Q}/{\Z}$. 
We write $\rho_{A}=\rho_{L}$ when $A=A_{L}$ for an even lattice $L$. 

A ${\C}A$-valued holomorphic function $f$ on ${\HH}$ is called a 
\textit{weakly holomorphic modular form} 
of type $\rho_{A}$ and weight $k\in\frac{1}{2}{\Z}$ if 
\begin{equation*}
f(M\tau ) = \phi(\tau)^{2k} \rho_{A}(\gamma) f(\tau) 
\end{equation*}
for every $\gamma=(M, \phi)\in{\Mp}$ and $f$ is meromorphic at the cusp. 
We write 
\begin{equation*}
f(\tau) = 
\sum_{\lambda\in A} f_{\lambda}(\tau){\elambda} = 
\sum_{\lambda\in A} \sum_{n\in q(\lambda)+{\Z}} c_{\lambda}(n) q^n {\elambda} 
\end{equation*}
for its Fourier expansion,  
where 
$q^n={\rm exp}(2\pi i n\tau)$ for $n\in{\Q}$. 
The finite sum 
$\sum_{\lambda} \sum_{n\leq 0} c_{\lambda}(n) q^n{\elambda}$ 
is called the \textit{principal part} of $f$. 
We say that $f$ has \textit{integral principal part} 
if all the Fourier coefficients $c_{\lambda}(n)$ with $n\leq0$ are integers. 
We write $M_{k}^{!}(\rho_{A})$ for the space of 
weakly holomorphic modular forms of weight $k$ and type $\rho_{A}$.

Theta series provide basic examples of holomorphic modular forms for the Weil representations. 
Let $K$ be an even positive-definite lattice. 
For $\lambda\in A_K$ the theta series $\theta_{K+\lambda}(\tau)$ is defined by 
\begin{equation*}
\theta_{K+\lambda}(\tau) = 
\sum_{l\in K+\lambda}q^{(l, l)/2} = 
\sum_{n\in q(\lambda)+{\Z}} c_{\lambda}^{K}(n)q^n, 
\end{equation*}
where $c_{\lambda}^{K}(n)$ is the number of vectors in 
$K+\lambda \subset K^{\vee}$ of norm $2n\geq0$. 
Note that $c_{\lambda}^{K}(n)$ is finite because $K^{\vee}$ is positive-definite. 
Taking the formal sum over $\lambda \in A_K$, we obtain the ${\C}A_K$-valued function 
\begin{equation*}
\Theta_K(\tau) = 
\sum_{\lambda\in A_K} \theta_{K+\lambda}(\tau){\elambda}. 
\end{equation*}
By Theorem 4.1 of \cite{Bo98}, 
this is a holomorphic modular form of type $\rho_K$ and weight ${\rm rk}(K)/2$ for ${\Mp}$.

\subsection{Three operations}\label{ssec:operations}

Borcherds found some operations for  
constructing modular forms for the Weil representations, 
which were subsequently developed by Bruinier and Scheithauer, 
as we recall below:   
\begin{itemize} 
\item Pullback to a sublattice (\cite{Br}, \cite{B-Y}, \cite{Br14}, \cite{Scheit15}) 
\item Pushforward to an overlattice (\cite{Bo98}, \cite{Br}, \cite{B-Y}, \cite{Br14}) 
\item Induction from scalar-valued modular forms 
(\cite{Bo00}, \cite{Scheit06}, \cite{Scheit15}) 
\end{itemize}

Let $A'$ be a finite quadratic module and 
let $I$ be an isotropic subgroup of $A'$. 
Then $A=I^{\perp}/I$ inherits the structure of a finite quadratic module. 
We have $|A|=|A'|/|I|^{2}$ and $\sigma(A)=\sigma(A')$. 
For example, when $A'=A_{L'}$ for an even lattice $L'$ 
and $L$ is an even overlattice of $L'$, 
then $I=L/L'$ is an isotropic subgroup of $A_{L'}$, 
and we have $A \simeq A_{L}$ naturally. 
Every isotropic subgroup of $A_{L'}$ arises in this way. 

Let $p: I^{\perp} \to A$ be the natural projection. 
We define linear maps 
\begin{equation*}\label{eqn:pull and push*}
{\pullA}: {\C}A \to{\C}A', \qquad  
{\pushA}: {\C}A' \to {\C}A, 
\end{equation*} 
by 
\begin{equation*}
{\elambda}{\pullA}=\sum_{\mu\in p^{-1}(\lambda)}{\emu}, \qquad 
{\emu}{\pushA} = 
\begin{cases}
\mathbf{e}_{p(\mu)}, & \mu\in I^{\perp}, \\ 
0, & \mu\not\in I^{\perp},  
\end{cases}
\end{equation*}
for $\lambda\in A$ and $\mu\in A'$ respectively. 
We write 
\begin{equation}\label{eqn:pull and push}
{\pull}={\pullA}, \qquad {\push}={\pushA}, 
\end{equation} 
when $A'=A_{L'}$ and $I=L/L'$ as above. 

\begin{lemma}\label{lem:pull push Weil equiv} 
The linear maps ${\pullA}$ and ${\pushA}$ are equivariant 
with respect to the Weil representations $\rho_{A}$, $\rho_{A'}$. 
\end{lemma} 

This is well-known on the level of modular forms (see Corollary \ref{cor:pull push modular} below). 
Here we work at the level of representations, for which the proof is similar. 
We give the proof for the sake of completeness. 

\begin{proof}
It suffices to verify that 
$\rho_{A'}(\gamma)\circ {\pullA} = {\pullA} \circ \rho_{A}(\gamma)$ 
and 
$\rho_{A}(\gamma)\circ {\pushA} = {\pushA} \circ \rho_{A'}(\gamma)$ 
for $\gamma=T$ and $S$. 
The case $\gamma=T$ is evident. 
We check the case $\gamma=S$. 
Write 
$\zeta=e(-\sigma(A)/8)=e(-\sigma(A')/8)$. 
First, we consider ${\pushA}$. 
For $\mu\in A'$ we have 
\begin{equation*}
(\rho_{A'}(S)({\emu})){\pushA} 
= \frac{\zeta}{\sqrt{|A'|}} \sum_{\mu'\in I^{\perp}} 
e(-(\mu, \mu'))\mathbf{e}_{p(\mu')}. 
\end{equation*}
When $\mu\in I^{\perp}$, 
$(\mu, \mu')=(p(\mu), p(\mu'))$ depends only on $p(\mu')\in I^{\perp}/I$, 
so this is equal to 
\begin{equation*}
\frac{\zeta}{\sqrt{|A'|}} \cdot |I| \cdot 
\sum_{\lambda \in I^{\perp}/I} e(-(p(\mu), \lambda)){\elambda} 
= \rho_{A}(S)(\mathbf{e}_{p(\mu)}) 
= \rho_{A}(S)({\emu}{\pushA}). 
\end{equation*}
When $\mu\not\in I^{\perp}$, we have 
$\sum_{\mu'\in \mu_{0}+I} e(-(\mu, \mu'))=0$ for every $\mu_{0}\in A'$. 
Considering the division of $I^{\perp}$ into $I$-orbits, we obtain  
$(\rho_{A'}(S)({\emu})){\pushA}=0$. 
Hence   
$\rho_{A}(S)\circ {\pushA} = {\pushA} \circ \rho_{A'}(S)$. 

Next, we consider ${\pullA}$. 
For $\lambda \in A$ we have 
\begin{equation*}
\rho_{A'}(S)({\elambda}{\pullA}) 
 =  
\rho_{A'}(S)\left( \sum_{\mu\in p^{-1}(\lambda)}{\emu} \right)  
 =  
\frac{\zeta}{\sqrt{|A'|}} \sum_{\mu'\in A'}
\left( \sum_{\mu\in p^{-1}(\lambda)} e(-(\mu, \mu')) \right) \mathbf{e}_{\mu'}.  
\end{equation*}
Since $p^{-1}(\lambda)$ is an $I$-orbit, we have as above  
\begin{equation*}
\sum_{\mu\in p^{-1}(\lambda)}e(-(\mu, \mu')) = 
\begin{cases}
|I|\cdot e(-(\lambda, p(\mu'))), \quad \mu'\in I^{\perp}, \\ 
\quad 0, \qquad  \qquad  \qquad \;   \mu'\not\in I^{\perp}. 
\end{cases}
\end{equation*}
It follows that 
\begin{equation*}
\rho_{A'}(S)({\elambda}{\pullA}) = 
\frac{\zeta}{\sqrt{|A|}} \sum_{\mu' \in I^{\perp}} 
e(-(\lambda, p(\mu'))) \mathbf{e}_{\mu'} = 
(\rho_{A}(S)({\elambda})){\pullA}. 
\end{equation*}
\end{proof}

The map ${\pullA}$ transforms ${\C}A$-valued functions to ${\C}A'$-valued functions, 
and ${\pushA}$ transforms ${\C}A'$-valued functions to ${\C}A$-valued functions. 
We denote these operators also by ${\pullA}$, ${\pushA}$. 
Lemma \ref{lem:pull push Weil equiv} implies the following. 

\begin{corollary}[\cite{Bo98}, \cite{Br}, \cite{B-Y}, \cite{Br14}, \cite{Scheit15}]\label{cor:pull push modular}
The operators ${\pullA}$, ${\pushA}$ define linear maps 
${\pullA}: M_{k}^{!}(\rho_A)\to M_{k}^{!}(\rho_{A'})$ and 
${\pushA}: M_{k}^{!}(\rho_{A'})\to M_{k}^{!}(\rho_A)$. 
\end{corollary} 


We now turn to describing induction from scalar-valued modular forms (\cite{Bo00}, \cite{Scheit06}). 
Let $A$ be a finite quadratic module. 
Let $d$ be a natural number divisible by the level of $A$. 
We write ${\MGd}$ for the inverse image of $\Gamma_0(d)$ in ${\Mp}$. 
By \cite{Scheit09}, \cite{Stro}, \cite{Ze}, there is a character $\chi_{A}$ of ${\MGd}$ such that 
$\rho_A(\gamma){\e}_0 = \chi_A(\gamma){\e}_0$ for every $\gamma\in{\MGd}$. 
More generally, if $I\subset A$ is an isotropic subgroup, we have  
\begin{equation*}
\rho_A(\gamma)\left(  \sum_{\lambda\in I} {\elambda} \right) = 
\chi_A(\gamma) \left(  \sum_{\lambda\in I} {\elambda} \right) 
\end{equation*}
for every $\gamma\in{\MGd}$ 
by Proposition 4.5 of \cite{Scheit09} and Lemma 5.6 of \cite{Stro}. 
See \S 4 of \cite{Scheit09}, \S 5 of \cite{Stro}, and \cite{Ze} for the explicit form of $\chi_A$. 
Now, if $\varphi$ is a scalar-valued weakly holomorphic modular form 
of weight $k$ and character $\chi_A$ for ${\MGd}$, 
we define  
\begin{equation}\label{eqn:def induction}
{\rm ind}_{A}^{I}(\varphi) = \sum_{\gamma\in{\MGd}\backslash{\Mp}} 
(\varphi|_{k}\gamma) \cdot \rho_A(\gamma)^{-1} \left( \sum_{\lambda\in I}{\elambda} \right), 
\end{equation}
where 
$(\varphi|_{k}\gamma)(\tau) = \phi(\tau)^{-2k}\varphi(M\tau)$ 
is the Petersson slash operator of weight $k$ by $\gamma=(M, \phi)$.  
This is a weakly holomorphic modular form of weight $k$ and type $\rho_A$ for ${\Mp}$. 
This construction is due to Borcherds (p.342 of \cite{Bo00}) for $I=\{ 0 \}$, 
and Scheithauer (Theorem 6.2 of \cite{Scheit06}) for general $I$. 
We especially denote ${\rm ind}_{A}^{\{ 0 \} }={\rm ind}_{A}$. 
When $A=A_{L}$ for an even lattice $L$, 
we also write $\chi_{L}=\chi_{A}$ and ${\rm ind}_{L}={\rm ind}_{A}$. 
Note that if $d_{A}$ is the level of $A$, we have 
\begin{equation*}
{\rm ind}_{A}^{I}(\varphi) = \sum_{\gamma\in M\Gamma_{0}(d_{A})\backslash{\Mp}} 
(\psi|_{k}\gamma) \cdot \rho_A(\gamma)^{-1} \left( \sum_{\lambda\in I}{\elambda} \right), 
\end{equation*}
where 
$\psi=
\sum_{\gamma \in {\MGd}\backslash M\Gamma_{0}(d_{A})}
(\varphi|_{k}\gamma)\chi_{A}(\gamma)^{-1}$ 
is the average of $\varphi$ over ${\MGd}\backslash M\Gamma_{0}(d_{A})$. 
In this sense, the induction is done essentially at level $d_{A}$. 

The relationship between ${\rm ind}_{A}$ and ${\pullA}$, ${\pushA}$ is as follows. 

\begin{lemma}\label{lem:induction and pull push}
Let $A'$ be a finite quadratic module and 
set $A=I^{\perp}/I$ for an isotropic subgroup $I$ of $A'$. 
Then for every natural number $d$ divisible by the level of $A'$, 
we have the equalities   
${\pullA}\circ {\rm ind}_{A}={\rm ind}_{A'}^{I}$ and  
${\pushA}\circ {\rm ind}_{A'}={\rm ind}_{A}$ 
on modular forms for $M\Gamma_{0}(d)$ with character $\chi_{A'}$. 
\end{lemma}

\begin{proof}
By Lemma \ref{lem:pull push Weil equiv} we have 
\begin{eqnarray*}
{\rm ind}_{A}(\varphi){\pullA} 
& = & 
\sum_{\gamma \in {\MGd}\backslash{\Mp}} 
(\varphi|_{k}\gamma)(\rho_A(\gamma)^{-1}\mathbf{e}_0){\pullA} \\ 
& = & 
\sum_{\gamma \in {\MGd}\backslash{\Mp}} 
(\varphi|_{k}\gamma) \rho_{A'}(\gamma)^{-1}(\mathbf{e}_0{\pullA})   
 =  
{\rm ind}_{A'}^{I}(\varphi). 
\end{eqnarray*}
We can verify the equality 
${\pushA}\circ {\rm ind}_{A'}={\rm ind}_{A}$ 
similarly, since ${\pushA}$ sends $\mathbf{e}_{0}\in{\C}A'$ to $\mathbf{e}_{0}\in{\C}A$. 
\end{proof}

Note that $\chi_{A}=\chi_{A'}$ over $M\Gamma_{0}(d_{A'})$ by Lemma \ref{lem:pull push Weil equiv}, 
so in particular $\chi_{A'}$ can be extended from $M\Gamma_{0}(d_{A'})$ to $M\Gamma_{0}(d_{A})$. 
Then 
$\rho_{A'}(\gamma)(\sum_{I}{\elambda})=\chi_{A}(\gamma)\sum_{I}{\elambda}$ 
for $\gamma\in M\Gamma_{0}(d_{A})$ 
by Lemma \ref{lem:pull push Weil equiv}. 
Hence the induction ${\rm ind}_{A'}^{I}$ can also be defined on 
$(M\Gamma_{0}(d_{A}), \chi_{A})$, 
not just on $(M\Gamma_{0}(d_{A'}), \chi_{A'})$, 
and Lemma \ref{lem:induction and pull push} holds also at level $d_{A}$.  
 

\begin{remark}[Jacobi form interpretation]\label{remark:pullback and Jacobi form}
Assume that $A=A_{N(-1)}$ for an even positive-definite lattice $N$. 
We identify $A=A_{N}$ as abelian groups naturally. 
Then $\rho_A$-valued modular forms 
$f(\tau)=\sum_{\lambda}f_{\lambda}(\tau){\elambda}$ 
correspond to Jacobi forms 
$\varphi(\tau, z)=\sum_{\lambda}f_{\lambda}(\tau)\theta_{N+\lambda}(\tau, z)$ 
of index $N$, 
where $\theta_{N+\lambda}(\tau, z)$ is the Jacobi theta series of $N+\lambda$ 
defined on $\mathbb{H}\times N_{{\C}}$ 
(see Lemma 2.3 of \cite{Gr95}).   

If $A'=A_{N'(-1)}$ for a finite-index sublattice $N'$ of $N$ and $I=N(-1)/N'(-1)$,  
the Jacobi form of index $N'$ corresponding to $f{\pullA}$ is 
just the same function $\varphi(\tau, z)$, 
considered on $\mathbb{H}\times N'_{{\C}}$ via the identification $N_{{\C}}=N'_{{\C}}$. 
This follows from the decomposition  
$$
\theta_{N+\lambda}(\tau, z) = \sum_{\mu\in p^{-1}(\lambda)}\theta_{N'+\mu}(\tau, z) 
$$
of the Jacobi theta series. 
Thus the interpretation of the operation ${\pullA}$ in terms of Jacobi forms is 
"changing the reference lattice", without changing the Jacobi form itself. 

If we restrict the Jacobi form $\varphi(\tau, z)$ to $\mathbb{H}\times \{ 0 \}$, 
we obtain the scalar-valued modular form 
$\varphi(\tau, 0) = \sum_{\lambda}f_{\lambda}(\tau)\theta_{N+\lambda}(\tau)$ 
because $\theta_{N+\lambda}(\tau, 0)=\theta_{N+\lambda}(\tau)$. 
This operation, replacing ${\elambda}$ by $\theta_{N+\lambda}(\tau)$ in  
$f=\sum_{\lambda}f_{\lambda}{\elambda}$ 
after $A_{N(-1)}\simeq A_{N}$, 
is the simplest example of the $\Theta$-contraction defined in \S \ref{ssec:split case}. 
More generally, 
when $N$ splits as $N_{1}\oplus N_{2}$, 
the Jacobi theta series decomposes as 
\begin{equation*}
\theta_{N+\lambda}(\tau, z) = 
\theta_{N_{1}+\lambda_{1}}(\tau, z_{1}) \cdot \theta_{N_{2}+\lambda_{2}}(\tau, z_{2}), 
\end{equation*}
where $z=(z_{1}, z_{2})$ with $z_{i}\in (N_{i})_{{\C}}$ and 
$\lambda=(\lambda_{1}, \lambda_{2})$ with $\lambda_{i}\in A_{i}:=A_{N_{i}(-1)}$. 
Therefore the restriction of $\varphi(\tau, z)$ 
to $\mathbb{H}\times (N_{1})_{{\C}}$ is given by   
\begin{equation*}
\varphi(\tau, z_{1}, 0) = 
\sum_{\lambda_{1}\in A_{1}} \sum_{\lambda_{2}\in A_{2}} 
f_{\lambda_{1}, \lambda_{2}}(\tau) \theta_{N_{2}+\lambda_{2}}(\tau) 
\theta_{N_{1}+\lambda_{1}}(\tau, z_{1}). 
\end{equation*}
This is the Jacobi form of index $N_{1}$ corresponding to 
the $\rho_{A_{1}}$-valued form  
\begin{equation}\label{eqn: Theta-contraction = restriction}
\sum_{\lambda_{1}\in A_{1}} \left( \sum_{\lambda_{2}\in A_{2}} 
f_{\lambda_{1}, \lambda_{2}}(\tau) \theta_{N_{2}+\lambda_{2}}(\tau) \right) 
\mathbf{e}_{\lambda_{1}}. 
\end{equation}
This operation on $f$ is a typical example of a $\Theta$-contraction. 
\end{remark}

\subsection{Borcherds products}\label{ssec:Borcherds prod}

Let $L$ be an even lattice of signature $(2, b)$. 
We recall the basic theory of Borcherds products for $L$ 
(see \S 13 of \cite{Bo98} and \S 3.3, \S 3.4 of \cite{Br} for more details). 
Let ${\DL}$ be the Hermitian symmetric domain attached to $L$, 
which is defined as 
one of the two connected components of the following open set of the isotropic quadric:  
\begin{equation*}\label{eqn:def type IV domain}
\{ [\omega] \in {\proj}(L_{{\C}}) \; | \; (\omega, \omega)=0, (\omega, \bar{\omega})>0  \} . 
\end{equation*}
We write ${\Or}(L)$ for the subgroup of ${\rm O}(L)$ preserving ${\DL}$, 
and ${\GL}$ for the kernel of the natural map ${\Or}(L)\to{\OAL}$. 

Let $\mathcal{O}(-1)$ be the tautological line bundle over ${\DL}$. 
Let $\Gamma$ be a finite-index subgroup of ${\Or}(L)$ and 
let $\chi: \Gamma \to {\C}^{\times}$ be a character. 
The group $\Gamma$ acts on $\mathcal{O}(-1)$ equivariantly. 
A meromorphic section $\Psi$ of $\mathcal{O}(-k)$ over ${\DL}$ satisfying 
$\gamma^{\ast}\Psi=\chi(\gamma)\Psi$ for every $\gamma\in \Gamma$ 
is called a \textit{meromorphic modular form} 
of weight $k$ and character $\chi$ with respect to $\Gamma$. 
If 
\begin{equation*}
\mathcal{D}_{L}^{\bullet} = 
\{  \omega \in L_{{\C}} \: | \: \omega \ne 0, \: [\omega]\in {\DL}  \} 
\end{equation*} 
is the affine cone over ${\DL}$ minus the vertex 
(which is the total space of $\mathcal{O}(-1)$ minus the zero section), 
a section of $\mathcal{O}(-k)$ over ${\DL}$ corresponds canonically to 
a function on $\mathcal{D}_{L}^{\bullet}$ that is homogeneous of degree $-k$ 
on every ${\C}^{\times}$-fiber of $\mathcal{D}_{L}^{\bullet}\to\mathcal{D}_{L}$. 
Thus a meromorphic modular form of weight $k$ and character $\chi$ 
is canonically identified with a meromorphic function $\Psi$ on $\mathcal{D}_{L}^{\bullet}$ 
which satisfies 
$\Psi(\gamma\omega)=\chi(\gamma)\Psi(\omega)$ for every $\gamma\in\Gamma$ and  
$\Psi(t\omega)=t^{-k}\Psi(\omega)$ for every $t \in {\C}^{\times}$ 
(cf.~\cite{Bo98}, \cite{Br}).

A vector $l\in L^{\vee}$ of negative norm defines 
the hyperplane section $l^{\perp}\cap{\DL}$ of ${\DL}$. 
This is called a \textit{rational quadratic divisor}  
and is naturally identified with $\mathcal{D}_{l^{\perp}\cap L}$. 
More generally, if $K(-1)$ is a negative-definite sublattice of $L$, 
the intersection $K(-1)^{\perp}\cap{\DL}$ is identified with ${\DM}$ where $M=K(-1)^{\perp}\cap L$. 
For $\lambda \in A_L$ and $n\in q(\lambda)+{\Z}$ with $n<0$, 
the locally finite divisor  
\begin{equation*}
Z(\lambda, n) = \sum_{\begin{subarray}{c} l\in L+\lambda \\ q(l)=n\end{subarray}} (l^{\perp}\cap{\DL}) 
\end{equation*} 
of ${\DL}$ is called the \textit{Heegner divisor} of discriminant $(\lambda, n)$. 
It descends to a finite divisor on ${\GL}\backslash {\DL}$. 
If $2\lambda\ne0$, every component of $Z(\lambda, n)$ has multiplicity $1$, 
while if $2\lambda=0$, the components have multiplicity $2$ 
because of the contribution from both $l$ and $-l$.  

\begin{theorem}[Borcherds \cite{Bo98}]
Let $f(\tau)=\sum_{\lambda}\sum_{n}c_{\lambda}(n)q^n{\elambda}$ 
be a weakly holomorphic modular form of type $\rho_L$ and weight $1-b/2$ 
with integral principal part and $c_0(0)\in2{\Z}$. 
Then there exists a meromorphic modular form $\Psi_L(f)$ on ${\DL}$ 
of weight $c_0(0)/2$ and some unitary character $\chi$ 
with respect to ${\GL}$ whose divisor is  
\begin{equation}\label{eqn:divisor Borcherds prod I}
{\divi}(\Psi_L(f)) = 
\frac{1}{2} \sum_{\lambda\in A_L} 
\sum_{\begin{subarray}{c} n<0 \\ n\in q(\lambda)+{\Z} \end{subarray}} 
c_{\lambda}(n)Z(\lambda, n). 
\end{equation}
\end{theorem}

The modular form $\Psi_L(f)$ is called the \textit{Borcherds product} associated to $f$. 
Equation \eqref{eqn:divisor Borcherds prod I} can also be written in the form  
\begin{eqnarray}\label{eqn:divisor Borcherds prod II}
{\divi}(\Psi_L(f)) 
&=& \frac{1}{2} \sum_{\begin{subarray}{c} l\in L^{\vee} \\ q(l)<0 \end{subarray}} c_{l+L}(q(l)) \, (l^{\perp}\cap{\DL}) \nonumber \\ 
&=& \sum_{\begin{subarray}{c} l\in L^{\vee}/\pm1 \\ q(l)<0 \end{subarray}} c_{l+L}(q(l)) \, (l^{\perp}\cap{\DL}). 
\end{eqnarray}
Here 
we have $c_{\lambda}(n)=c_{-\lambda}(n)$ by the invariance of $f$ under $Z=S^2$. 
The factor $1/2$ in Equation \eqref{eqn:divisor Borcherds prod I} arises 
from the multiplicities of the Heegner divisors $Z(\lambda, n)$.


\section{Quasi-pullbacks}\label{sec:quasi-pullback}

Let $L$ be an even lattice of signature $(2, b)$. 
Let $K(-1)$ be a primitive negative-definite sublattice of $L$ 
where $K$ is positive-definite. 
We assume that 
the Witt index of the orthogonal complement $M=K(-1)^{\perp}\cap L$ is 
smaller than ${\rm rk}(M)-2$. 
Since $M$ has signature $(2, \ast)$, 
its Witt index cannot exceed $2$, 
so this condition is always satisfied when ${\rm rk}(M)\geq 5$; 
when ${\rm rk}(M)=4$, this is equivalent to the absence of isotropic sublattices of rank $2$ in $M$ 
(e.g., the case of Hilbert modular surfaces); 
when ${\rm rk}(M)=3$, $M$ is required to be anisotropic,  
which is equivalent to $\Gamma\backslash\mathcal{D}_{M}$ having no cusps 
and hence being compact. 
Under this condition, we can use the Koecher principle on $\mathcal{D}_{M}$ in the following form. 

\begin{lemma}\label{lem:Koecher}
Let $M$ be as above. 
Let $\Gamma$ be a finite-index subgroup of ${\Or}(M)$ and 
let $\chi$ be a unitary character of $\Gamma$. 
Then any nonzero modular form of weight $0$ and character $\chi$ for $\Gamma$ 
which has no pole on $\mathcal{D}_{M}$ is constant 
(and we must have $\chi=1$). 
\end{lemma}

\begin{proof}
When ${\rm rk}(M)\geq4$, we can apply 
the Margulis normal subgroup theorem (Theorem 4' in p.4 of \cite{Margu}). 
Indeed, the Lie group ${\rm O}(M_{{\R}})$ is simple when ${\rm rk}(M)\geq 5$, 
while when ${\rm rk}(M)=4$ the discrete subgroup $\Gamma$ of ${\rm O}(M_{{\R}})$ is 
still irreducible by the Witt index condition. 
This tells us that the abelianization of $\Gamma$ is finite, 
so $\chi$ must have finite order. 
Thus, by passing to ${\rm Ker}(\chi)$, 
we are reduced to the usual Koecher principle for scalar-valued modular forms. 
When ${\rm rk}(M)=3$, we argue differently. 
If we pass to a torsion-free subgroup $\Gamma'$ of $\Gamma$ of finite index, 
$\chi|_{\Gamma'}$ corresponds to a line bundle of degree $0$ on the compact curve 
$\Gamma'\backslash\mathcal{D}_{M}$. 
Then we are reduced to the fact that 
no line bundle of degree $0$ on a compact curve except the trivial one 
has a nonzero holomorphic section, 
and every holomorphic section of the trivial line bundle is constant. 
\end{proof}

Let $f$ be a weakly holomorphic modular form of weight $1-b/2$ and type $\rho_L$ 
with integral principal part and $c_0(0)\in2{\Z}$,  
and let $\Psi=\Psi_L(f)$ be its Borcherds lift with unitary character $\chi$. 
For each primitive vector $l$ of $K(-1)$, 
we denote by $r(l)$ the order of $\Psi$ along the rational quadratic divisor $l^{\perp}\cap{\DL}$. 
The \textit{quasi-pullback} of $\Psi$ to ${\DM}$ is defined 
(see, e.g., pp.~200, 210 of \cite{Bo95} and p.~188 of \cite{B-K-P-SB}) 
by 
\begin{equation*}
\Psi||_{{\DM}} = 
 \left. \frac{\Psi}{\prod_{\pm l}(\cdot, l)^{r(l)}} \: \right|_{{\DM}},  
\end{equation*}
where $\pm l\in K(-1)/\pm1$ runs over \textit{primitive} vectors of $K(-1)$ up to $\pm1$, 
and $(\cdot, l)$ is the linear form on $\mathcal{D}_{L}^{\bullet}\subset L_{{\C}}$ 
defined by the pairing with the vector $l$.  
Here, for each $[\pm l]$ from $K(-1)/\pm1$, we choose either $l$ or $-l$ as a representative 
and take the linear form with it  
(so there is in general a choice of $\pm1$ when defining $\Psi||_{{\DM}}$).  
Note that the product $\prod_{\pm l}(\cdot, l)^{r(l)}$ is actually a finite product. 
Indeed, since $f$ has only finitely many nonzero Fourier coefficients $c_{\lambda}(n)$ with $n<0$ 
and since the negative-definite lattice $K(-1)_{{\Q}}\cap L^{\vee}$ 
contains only finitely many vectors of a given norm, 
we have $r(l)\ne0$ only for finitely many primitive $l\in K(-1)$. 

\begin{lemma}[\cite{Bo95}, \cite{B-K-P-SB}]\label{lem:def quasi-pullback}
The quasi-pullback $\Psi||_{{\DM}}$ is a nonzero meromorphic modular form on ${\DM}$ 
with respect to $\Gamma_{M}$ and the character $\chi|_{\Gamma_{M}}$, 
and has weight ${\rm wt}(\Psi)+\sum_{\pm l}r(l)$ where ${\rm wt}(\Psi)$ is the weight of $\Psi$.  
Here $\chi$ is restricted to $\Gamma_{M}$ via the natural embedding 
$\Gamma_{M}\hookrightarrow \Gamma_{L}$. 
\end{lemma}

\begin{proof}
We write $\Psi'=\Psi/\prod_{\pm l}(\cdot, l)^{r(l)}$ 
and $k={\rm wt}(\Psi)+\sum_{\pm l}r(l)$. 
By definition $\Psi'$ is a meromorphic section of $\mathcal{O}(-k)$ over ${\DL}$. 
Since rational quadratic divisors on ${\DL}$ containing ${\DM}$ are 
exactly $l^{\perp}\cap{\DL}$ with $l\in K(-1)$, 
we find that  
\begin{equation*}
{\divi}(\Psi') = 
{\divi}(\Psi) - \sum_{\pm l} r(l)(l^{\perp}\cap {\DL}) 
\end{equation*}
does not contain ${\DM}$ in its support. 
Hence $\Psi||_{{\DM}}=\Psi'|_{{\DM}}$ is 
a nonzero meromorphic section of $\mathcal{O}(-k)|_{{\DM}}$. 

Nikulin shows in \cite{Ni} that for every $\gamma\in \Gamma_{M}$, 
the isometry $\tilde{\gamma}=\gamma\oplus {\rm id}_{K(-1)}$ of $M\oplus K(-1)$ 
extends to an isometry of $L$ and acts trivially on $A_L$. 
This defines an embedding $\Gamma_{M}\hookrightarrow{\GL}$. 
We have $\tilde{\gamma}^{\ast}\Psi=\chi(\tilde{\gamma})\Psi$,  
and also $\tilde{\gamma}$ leaves $\prod_{\pm l}(\cdot, l)^{r(l)}$ invariant 
because it fixes vectors $l$ in $K(-1)$. 
Therefore $\tilde{\gamma}^{\ast}\Psi'=\chi(\tilde{\gamma})\Psi'$ 
for every $\gamma \in \Gamma_{M}$. 
If we write $\chi'=\chi|_{\Gamma_{M}}$, 
then $\Psi||_{{\DM}}=\Psi'|_{{\DM}}$ satisfies 
$\gamma^{\ast}(\Psi||_{{\DM}})=\chi'(\gamma)\Psi||_{{\DM}}$ 
for every $\gamma\in \Gamma_{M}$. 
\end{proof}

Our purpose is to explicitly construct a weakly holomorphic modular form of type $\rho_M$ 
whose Borcherds lift gives $\Psi||_{{\DM}}$. 
In \S \ref{ssec:split case} we consider the split case $L=M\oplus K(-1)$. 
The general case is studied in \S \ref{ssec:general case}, 
where we prove Theorem \ref{main}. 
In \S \ref{ssec:theta contraction induction} we give a more explicit formula when 
$f$ is induced from a scalar-valued modular form. 
In \S \ref{ssec:example} we consider a few examples.

\subsection{The split case}\label{ssec:split case}

In this subsection we consider the case where $L$ splits as $M\oplus K(-1)$. 
We identify $A_{K(-1)}=A_K$ as abelian groups, 
which multiplies the discriminant form by $-1$. 
For $\lambda \in A_{K}=A_{K(-1)}$ we write 
$\mathbf{e}_{\lambda}\in {\C}A_{K}$ and $\bar{\mathbf{e}}_{\lambda}\in {\C}A_{K(-1)}$ 
for the respective corresponding vectors. 
We have a canonical isomorphism 
${\C}A_{K(-1)}\to ({\C}A_{K})^{\vee}$ 
sending $\bar{\mathbf{e}}_{\lambda}$ to the dual basis vector $\mathbf{e}_{\lambda}^{\vee}$ 
for each $\lambda \in A_{K}$. 
This is an isomorphism $\rho_{K(-1)}\simeq (\rho_{K})^{\vee}$ of ${\Mp}$-representations. 
Since $A_L=A_M\oplus A_{K(-1)}$, 
we have a natural isomorphism ${\C}A_L\simeq {\C}A_M\otimes{\C}A_{K(-1)}$ 
sending ${\e}_{(\mu, \lambda)}$ to ${\emu}\otimes \bar{\mathbf{e}}_{\lambda}$ 
where $\mu \in A_{M}$ and $\lambda \in A_{K(-1)}$. 
This is an isomorphism $\rho_L \simeq \rho_M \otimes \rho_{K(-1)}$ of ${\Mp}$-representations. 

Let $f$ be a ${\C}A_L$-valued function on ${\HH}$. 
By ${\C}A_L\simeq {\C}A_M\otimes{\C}A_{K(-1)}$ 
we view $f$ as a family of 
${\C}A_M$-valued functions parametrized by $A_{K(-1)}=A_{K}$, and write 
\begin{equation*}
f = \sum_{\lambda\in A_{K}} f_{\lambda}\otimes \bar{\mathbf{e}}_{\lambda} 
\end{equation*}
with $f_{\lambda}$ being a ${\C}A_M$-valued function. 
We define the \textit{$\Theta$-contraction} of $f$ as the ${\C}A_M$-valued function 
\begin{equation}\label{eqn:Theta contraction}
\langle f, {\ThetaK} \rangle = \sum_{\lambda \in A_K} f_{\lambda} \cdot \theta_{K+\lambda}. 
\end{equation} 
Equivalently, 
consider $f$ as ${\C}A_{M}\otimes({\C}A_{K})^{\vee}$-valued  
by the isomorphism 
${\C}A_{K(-1)}\simeq({\C}A_{K})^{\vee}$. 
Then $f\otimes{\ThetaK}$ is a ${\C}A_{M}\otimes({\C}A_{K})^{\vee}\otimes{\C}A_{K}$-valued function, 
and $\langle f, {\ThetaK} \rangle$ is obtained from $f\otimes{\ThetaK}$ 
by the contraction map $({\C}A_{K})^{\vee}\otimes{\C}A_{K}\to {\C}$.  

\begin{lemma}\label{lem:quasipullback WD}
If $f$ is a weakly holomorphic modular form of type $\rho_L$ and weight $k$, 
then $\langle f, {\ThetaK} \rangle$ is a weakly holomorphic modular form 
of type $\rho_M$ and weight $k+{\rk}(K)/2$. 
If $f$ has integral principal part, so does $\langle f, {\ThetaK} \rangle$. 
When $2k\equiv 2-b$ mod $4$, 
if furthermore the Fourier coefficient $c_0(0)$ of $f$ is even, 
then so is the constant term of $\langle f, {\ThetaK} \rangle$. 
\end{lemma}

\begin{proof}
Since ${\ThetaK}$ is a modular form of type $\rho_{K}$ and weight ${\rk}(K)/2$, 
the tensor product $f\otimes{\ThetaK}$ is modular of  
type $\rho_M\otimes(\rho_K)^{\vee}\otimes\rho_K$ and 
weight $k+{\rk}(K)/2$. 
Since the contraction map $(\rho_K)^{\vee}\otimes\rho_K \to {\C}$ is ${\Mp}$-invariant,  
$\langle f, {\ThetaK} \rangle$ is modular of type $\rho_M$ and weight $k+{\rk}(K)/2$. 
The second assertion follows from Equation \eqref{eqn:Theta contraction}, 
because $\theta_{K+\lambda}$ is holomorphic at the cusp 
and has integral Fourier coefficients. 
As for the last assertion, if 
$f_{\lambda}(\tau)=\sum_{\mu,n}c_{\mu,\lambda}(n)q^n{\emu}$ and 
$\theta_{K+\lambda}(\tau)=\sum_{m}c_{\lambda}^{K}(m)q^m$, 
then we have 
$c_{0,\lambda}(n)=c_{0,-\lambda}(n)$ 
due to the invariance under $Z$ 
and our assumption on the weight $k$, 
and we also have 
$c_{\lambda}^{K}(m)=c_{-\lambda}^{K}(m)$ 
due to the multiplication by $-1$. 
So the Fourier coefficient of $\langle f, {\ThetaK} \rangle$ 
at $q^{0}\mathbf{e}_{0}$ can be written as 
\begin{equation*}
c_{0,0}(0) + 
2 \sum_{m>0}\sum_{\begin{subarray}{c} \lambda \in A_{K}/\pm1 \\ 2\lambda \ne 0 \end{subarray}} 
c_{0,\lambda}(-m)c_{\lambda}^{K}(m) + 
\sum_{m>0}\sum_{\begin{subarray}{c} \lambda \in A_{K} \\ 2\lambda = 0\end{subarray}}  
c_{0,\lambda}(-m)c_{\lambda}^{K}(m). 
\end{equation*}
Since the multiplication by $-1$ preserves $K+\lambda$ if $2\lambda=0\in A_{K}$, 
we have $c_{\lambda}^{K}(m)\in 2{\Z}$ for such $\lambda$ and $m>0$. 
This proves our assertion. 
\end{proof} 

We can now prove our main result in the split case, 
from which the general case will follow later. 

\begin{proposition}\label{prop:formula split case}
Assume that $L$ splits as $M\oplus K(-1)$. 
Let $f$ be a weakly holomorphic modular form of weight $1-b/2$ and type $\rho_L$ 
with integral principal part and $c_0(0)\in2{\Z}$. 
Then we have 
$\Psi_L(f)||_{{\DM}} = \Psi_{M} (\langle f, {\ThetaK} \rangle)$ 
up to a multiplicative constant. 
\end{proposition}

\begin{proof}
We show that the two modular forms on $\mathcal{D}_{M}$ have the same weight and divisor. 
Then their ratio is a modular form of weight $0$ on $\mathcal{D}_M$ for a unitary character 
which has no pole on $\mathcal{D}_M$. 
By our Witt index condition on $M$, we can apply Lemma \ref{lem:Koecher} 
to see that this ratio is a constant.   

By Equation \eqref{eqn:divisor Borcherds prod II}, 
$\Psi_L(f)||_{{\DM}}$ can be written up to a constant as   
\begin{equation}\label{eqn: another expression of quasi-pullback}
\Psi_L(f)||_{{\DM}} = 
\left. \frac{\Psi_L(f)}{\prod_{\pm v}(\cdot, v)^{c_{v+L}(q(v))}} \: \right|_{{\DM}}, 
\end{equation}
where $v$ runs over \textit{all} the nonzero vectors of $K(-1)^{\vee}$ up to $\pm1$ 
(not necessarily primitive in $K(-1)$). 
The product here is again a finite product by the same argument as before. 
In order to compare this with the Borcherds lift of $\langle f, {\ThetaK} \rangle$, 
we calculate the Fourier coefficients of $\langle f, {\ThetaK} \rangle$. 
Write 
\begin{equation*}
f(\tau) 
 =  
\sum_{\mu\in A_M}\sum_{\lambda\in A_K} \sum_{n} 
c_{\mu,\lambda}^{L}(n)q^n {\emu}\otimes \bar{{\e}}_{\lambda}, 
\end{equation*}
\begin{equation*}
{\ThetaK}(\tau) 
 =  
\sum_{\lambda \in A_K} \sum_{m} c_{\lambda}^{K}(m)q^m {\elambda}, 
\end{equation*}
and 
\begin{equation*} 
\langle f, {\ThetaK} \rangle (\tau)
 =  
\sum_{\mu\in A_M} \sum_{l} c_{\mu}^{M}(l)q^l {\emu}. 
\end{equation*}
By the definition of $\langle f, {\ThetaK} \rangle$ in Equation \eqref{eqn:Theta contraction}, 
we have  
\begin{equation*}
c_{\mu}^M(l) 
 =  \sum_{\lambda \in A_K} \sum_{n+m=l} c_{\mu,\lambda}^L(n)c_{\lambda}^{K}(m)  
 =  \sum_{\lambda\in A_K}\sum_{m\geq0} c_{\lambda}^K(m)c_{\mu,\lambda}^L(l-m).   
\end{equation*}
Since $c_{\lambda}^K(m)$ is the number of vectors $v$ in $K(-1)+\lambda$ of norm $-2m$, 
we find that  
\begin{equation}\label{eqn:relation od Fourier coefficients}
c_{\mu}^M(l) = \sum_{v\in K(-1)^{\vee}}c_{(\mu,v)+L}^L(l+q(v)). 
\end{equation}
Note that this is a finite sum because $f$ is meromorphic at the cusp 
and $K(-1)^{\vee}$ is negative-definite. 
In particular, we have 
\begin{equation*}
c_0^M(0) = c_0^L(0) + \sum_{\begin{subarray}{c} v\in K(-1)^{\vee} \\ v\ne0 \end{subarray}} c_{v+L}^L(q(v)). 
\end{equation*}
Hence the weight of $\Psi_{M}(\langle f, {\ThetaK} \rangle)$ is  
\begin{equation*}
\frac{1}{2}c_0^L(0) + 
\frac{1}{2} \sum_{\begin{subarray}{c} v\in K(-1)^{\vee} \\ v\ne0 \end{subarray}} c_{v+L}^L(q(v)) 
= {\rm wt}(\Psi_{L}(f)) + 
\sum_{\begin{subarray}{c} v\in K(-1)^{\vee}/\pm1 \\ v\ne0 \end{subarray}} c_{v+L}^L(q(v)). 
\end{equation*}
By Equation \eqref{eqn: another expression of quasi-pullback}, 
this is equal to the weight of $\Psi_L(f)||_{{\DM}}$. 

We compare the divisors next. 
By Equations \eqref{eqn:divisor Borcherds prod II} and \eqref{eqn:relation od Fourier coefficients},  
the divisor of $\Psi_M(\langle f, {\ThetaK} \rangle)$ is given by 
\begin{eqnarray*}
{\divi}(\Psi_M(\langle f, {\ThetaK} \rangle)) 
& = & \sum_{\begin{subarray}{c} u\in M^{\vee}/\pm1 \\ q(u)<0 \end{subarray}} c_{u+M}^M(q(u)) (u^{\perp}\cap{\DM}) \\ 
& = & \sum_{\begin{subarray}{c} u\in M^{\vee}/\pm1 \\ q(u)<0 \end{subarray}} \sum_{v\in K(-1)^{\vee}}  
c_{(u,v)+L}^{L}(q(u+v)) (u^{\perp}\cap{\DM}).  
\end{eqnarray*}
If we write $\pi:L^{\vee}\to M^{\vee}$ for the projection, this can be written as 
\begin{equation*} 
\sum_{\begin{subarray}{c} u\in M^{\vee}/\pm1 \\ q(u)<0 \end{subarray}} 
\sum_{\begin{subarray}{c} w\in L^{\vee} \\ \pi(w)=u \end{subarray}}  
c_{w+L}^{L}(q(w)) (w^{\perp}\cap{\DM}) 
= 
\sum_{\begin{subarray}{c} w\in L^{\vee}/\pm1 \\ q(\pi(w))<0 \end{subarray}} 
c_{w+L}^{L}(q(w)) (w^{\perp}\cap{\DM}). 
\end{equation*}
Since $q(\pi(w))<0$ if and only if 
$w^{\perp}\cap\mathcal{D}_M\ne \emptyset$ 
and $w\notin K(-1)^{\vee}$, 
this equals the divisor of $\Psi_L(f)||_{{\DM}}$ 
by Equations \eqref{eqn:divisor Borcherds prod II} and \eqref{eqn: another expression of quasi-pullback}. 
\end{proof}


\begin{remark}\label{remark:theta contraction and Jacobi form}
The observation from Equation \eqref{eqn: Theta-contraction = restriction} above 
can be phrased in terms of $\Theta$-contraction: 
when $M$ splits as $U \oplus U\oplus N(-1)$ with $N$ positive-definite, 
the Jacobi form of index $N$ corresponding to $\langle f, \Theta_{K} \rangle$ 
is the restriction of the Jacobi form of index $N\oplus K$ corresponding to $f$. 
Thus the interpretation of $\Theta$-contraction in terms of Jacobi forms is "restriction". 
\end{remark}

\subsection{The general case}\label{ssec:general case}

Next, we consider the general case where 
$L$ does not necessarily coincide with $M\oplus K(-1)$, 
and prove Theorem \ref{main}. 
We need the following general lemma. 

\begin{lemma}\label{lem:pull on Borcherds product}
Let $L'$ be a finite-index sublattice of $L$. 
Then we have 
$\Psi_L(f) = \Psi_{L'}(f{\pull})$ 
up to a constant
under the natural identification ${\DL}=\mathcal{D}_{L'}$, 
where $\uparrow_{L}^{L'}$ is the pullback operation defined in 
Equation \eqref{eqn:pull and push}.  
\end{lemma}

\begin{proof}
We write the Heegner divisors as $Z(\lambda, n)_L$ and $Z(\mu, n)_{L'}$ 
in order to specify the reference lattice. 
We denote $I=L/L'\subset A_{L'}$ 
and $p:I^{\perp}\to A_L$ the projection. 
Since $L=\sqcup_{\mu\in I}(L'+\mu)$, we have the disjoint decomposition 
\begin{equation*}
L+\lambda = 
\bigsqcup_{\begin{subarray}{c} \mu\in I^{\perp} \\ p(\mu)=\lambda \end{subarray}} 
(L'+\mu) 
\end{equation*}
for $\lambda\in A_{L}$. 
Hence the Heegner divisors for $L$ decompose as  
\begin{equation*}
Z(\lambda, n)_L = \sum_{\begin{subarray}{c} \mu \in I^{\perp} \\ p(\mu)=\lambda \end{subarray}} 
Z(\mu, n)_{L'}. 
\end{equation*}
It follows that 
\begin{eqnarray*}
{\divi}(\Psi_L(f)) 
& = & 
\frac{1}{2} \sum_{\lambda \in A_L} \sum_{\begin{subarray}{c} n<0 \\ n\equiv q(\lambda) \end{subarray}} 
c_{\lambda}(n) Z(\lambda, n)_L \\ 
& =  &
\frac{1}{2} \sum_{\mu \in I^{\perp}} \sum_{\begin{subarray}{c} n<0 \\ n\equiv q(\mu) \end{subarray}} 
c_{p(\mu)}(n) Z(\mu, n)_{L'}. 
\end{eqnarray*}
On the other hand, since the Fourier expansion of $f{\pull}$ is given by 
\begin{equation*}
(f{\pull})(\tau) = \sum_{\mu\in I^{\perp}} \sum_{n\equiv q(\mu)} 
c_{p(\mu)}(n) q^n {\emu}, 
\end{equation*}
we see that $\Psi_{L}(f)$ and $\Psi_{L'}(f{\pull})$ have the same divisor on ${\DL}=\mathcal{D}_{L'}$. 
Since $f$ and $f{\pull}$ have the same coefficient of $q^{0}\mathbf{e}_{0}$, 
they also have the same weight. 
\end{proof}

We now prove the main result of this paper. 

\begin{proof}[(Proof of Theorem \ref{main})]
We apply Lemma \ref{lem:pull on Borcherds product} to $L'=M\oplus K(-1)$. 
The modular form $\Psi_L(f)||_{{\DM}}$ can be obtained by 
first considering $\Psi_L(f)$ as a modular form on $\mathcal{D}_{L'}$ 
and then taking its quasi-pullback from $\mathcal{D}_{L'}$ to ${\DM}$. 
We thus have 
\begin{equation*}
\Psi_L(f)||_{{\DM}} = 
\Psi_{L'}(f{\pull})||_{{\DM}} = 
\Psi_M( \langle f{\pull}, {\ThetaK} \rangle ) 
\end{equation*}
by Lemma \ref{lem:pull on Borcherds product} and Proposition \ref{prop:formula split case}. 
\end{proof}

\begin{remark}
Assume that $M$ contains $U \oplus U$ and  
write $M=U \oplus U\oplus N(-1)$ and $L=U \oplus U\oplus N_{0}(-1)$. 
By Remarks \ref{remark:pullback and Jacobi form} and 
\ref{remark:theta contraction and Jacobi form}, 
the Jacobi form of index $N$ corresponding to 
$\langle f{\pull}, \Theta_{K} \rangle$ 
is the restriction of the Jacobi form of index $N_{0}$ corresponding to $f$. 
Thus we obtain another proof of Gritsenko's result that when $M$ contains $U \oplus U$, 
the quasi-pullback of the Borcherds lift of a weak Jacobi form (of weight $0$)  
is the Borcherds lift of the restriction of this Jacobi form. 
See pp.16, 21, 23 of \cite{Gr10} for some examples of Gritsenko's formula. 
\end{remark}

\subsection{$\Theta$-contraction and induction}\label{ssec:theta contraction induction}

In the case where the $\rho_L$-valued form $f$ is constructed as the induction ${\ind}(\varphi)$ 
from a scalar-valued modular form $\varphi$ 
(cf.~\cite{Bo00}, \cite{Scheit06}, \cite{Scheit09}, \cite{Yo09}, \cite{Yo13}, \cite{Scheit15}), 
we can describe the $\rho_M$-valued form $\langle f{\pull}, {\ThetaK} \rangle$ more explicitly.    
Let $L'=M\oplus K(-1)$ and 
\begin{equation}\label{eqn:I = L/(M+K)}
I = L/L' \subset A_{L'} = A_M\oplus A_{K(-1)}. 
\end{equation}
This is an isotropic subgroup of $A_M\oplus A_{K(-1)}$. 
Let $G_M\subset A_{M}$ and $G_{K}\subset A_{K}$ be the images of $I$ 
by the projections $I\to A_{M}$ and $I\to A_{K(-1)}=A_{K}$, respectively. 
Nikulin shows in \cite{Ni} that these projections are injective and so $I$ is the graph of an isomorphism 
\begin{equation}\label{eqn:iota:GMtoGK}
\iota:G_M\to G_K. 
\end{equation}
This $\iota$ is an isometry because $I$ is isotropic and we take the $(-1)$-scaling $A_{K(-1)}=A_{K}$. 

We fix a natural number $d$ divisible by the level of $A_L$. 
We choose and fix representatives $\gamma_{1}, \cdots, \gamma_{a}\in{\Mp}$ 
of ${\MGd}\backslash{\Mp}$. 
For an element $\mu \in A_M$ and 
a modular form $\psi$ of weight $k'$ for \textit{some} subgroup of ${\Mp}$, we define 
\begin{equation*}
{\rm ind}^{\mu}_{M}(\psi) = 
\sum_{i=1}^{a} (\psi|_{k'}\gamma_{i}) \rho_M(\gamma_{i})^{-1}({\emu}). 
\end{equation*}
If $\mu=0$, the level of $A_M$ divides $d$, 
and $\psi$ is modular for ${\MGd}$ with character $\chi_{M}$, 
then this is the operation ${\rm ind}_M$ defined in Equation \eqref{eqn:def induction}. 
But in general this may depend on the choice of the representatives $\gamma_{1}, \cdots, \gamma_{a}$.

\begin{lemma}\label{prop:induction and theta contraction non-split case}
Let $\varphi$ be a weakly holomorphic scalar-valued modular form 
of weight $k$ and character $\chi_{L}$ for ${\MGd}$. 
Then for $L'=M\oplus K(-1)$ we have 
\begin{equation}\label{eqn:induction and theta contraction}
\langle {\ind}(\varphi){\pull}, {\ThetaK} \rangle = 
\sum_{\mu \in G_M} {\rm ind}^{\mu}_{M}(\varphi \cdot \theta_{K+\iota(\mu)}).  
\end{equation}
\end{lemma}

\begin{proof}
By Lemma \ref{lem:induction and pull push} we have 
\begin{equation*}
{\ind}(\varphi){\pull}
 =   
{\rm ind}_{L'}^{I}(\varphi) 
 =  
\sum_{i=1}^{a}(\varphi|_{k}\gamma_{i}) \sum_{\mu \in G_M}
(\rho_M(\gamma_{i})^{-1}{\emu} )\otimes (\rho_{K}^{\vee}(\gamma_{i})^{-1}\bar{{\e}}_{\iota(\mu)}). 
\end{equation*}
Hence   
\begin{equation*}
\langle {\ind}(\varphi){\pull}, \, {\ThetaK} \rangle = 
\sum_{i=1}^{a} (\varphi|_{k}\gamma_{i}) \sum_{\mu \in G_M} 
(\rho_M(\gamma_{i})^{-1}{\emu}) \cdot 
\langle \rho_{K}^{\vee}(\gamma_{i})^{-1}\bar{{\e}}_{\iota(\mu)}, \, {\ThetaK} \rangle. 
\end{equation*}
By the modularity of ${\ThetaK}$, we have  
\begin{equation*}
\langle \rho_{K}^{\vee}(\gamma)^{-1}\bar{{\e}}_{\lambda},  {\ThetaK} \rangle = 
\langle \bar{{\e}}_{\lambda}, \, \rho_{K}(\gamma){\ThetaK} \rangle = 
\langle \bar{{\e}}_{\lambda},   {\ThetaK}|_{\kappa}\gamma \rangle = 
\theta_{K+\lambda}|_{\kappa}\gamma 
\end{equation*}
for every $\gamma\in{\Mp}$ and $\lambda\in A_{K}$, 
where $\kappa={\rm rk}(K)/2$. 
It follows that  
\begin{eqnarray*}
\langle {\ind}(\varphi){\pull}, \: {\ThetaK} \rangle 
& = & 
\sum_{i=1}^{a}(\varphi|_{k}\gamma_{i}) 
\sum_{\mu \in G_M}(\theta_{K+\iota(\mu)}|_{\kappa}\gamma_{i}) \rho_M(\gamma_{i})^{-1}{\emu}  \\ 
& = & 
\sum_{\mu \in G_M} {\rm ind}^{\mu}_M(\varphi \cdot \theta_{K+\iota(\mu)}).     
\end{eqnarray*}
\end{proof}

Equation \eqref{eqn:induction and theta contraction} implies that 
the sum in the right hand side does not depend on the choice of $\gamma_{1}, \cdots , \gamma_{a}$. 

By Lemma \ref{prop:induction and theta contraction non-split case} 
and Theorem \ref{main} we obtain the following. 

\begin{proposition}\label{prop:induction and theta contraction}
Let $\varphi$ be a scalar-valued weakly holomorphic modular form 
of weight $1-b/2$ and character $\chi_L$ for ${\MGd}$, 
such that ${\ind}(\varphi)$ has integral principal part and $c_0(0)\in2{\Z}$. 
Then 
\begin{equation*}
\Psi_L({\ind}(\varphi))||_{{\DM}} = 
\Psi_M \left( \sum_{\mu\in G_M} 
{\rm ind}^{\mu}_{M}(\varphi \cdot \theta_{K+\iota(\mu)}) \right). 
\end{equation*}
\end{proposition}


When $L=M\oplus K(-1)$, we have $I=\{ 0 \}$, 
so Equation \eqref{eqn:induction and theta contraction} takes the simple form 
\begin{equation*}
\langle {\ind}(\varphi), {\ThetaK} \rangle = {\rm ind}_{M}(\varphi \cdot \theta_{K}).  
\end{equation*}
Hence Proposition \ref{prop:induction and theta contraction} is simplified as follows. 

\begin{corollary}\label{cor:induction and theta contraction split case}
When $L$ splits as $M\oplus K(-1)$, we have  
\begin{equation*}
\Psi_L({\ind}(\varphi))||_{{\DM}} = \Psi_M({\rm ind}_{M}(\varphi \cdot \theta_{K})). 
\end{equation*}
\end{corollary}

\subsection{Examples}\label{ssec:example}

We discuss a few examples. 
In what follows, $L$ is always an even lattice of signature $(2, b)$, 
$K(-1)$ is a primitive negative-definite sublattice of $L$, 
$M=K(-1)^{\perp}\cap L$ 
(satisfying the assumption about its rank and Witt index 
appearing in the beginning of \S \ref{sec:quasi-pullback}), 
and $L'=M\oplus K(-1)$. 
We use the notation  
$I=L/L'$ and $\iota:G_M\to G_K$ 
from Equations \eqref{eqn:I = L/(M+K)} and \eqref{eqn:iota:GMtoGK}. 

\begin{example}\label{ex1}
Consider the case where $G_{K}\subset A_{K}$ is nondegenerate. 
We have the orthogonal decompositions 
$A_{K}=G_{K}\oplus G_{K}^{\perp}$ and 
$A_{M}=G_{M}\oplus G_{M}^{\perp}$. 
Therefore 
$I^{\perp}=G_{M}^{\perp} \oplus G_{K}^{\perp}(-1) \oplus I$, 
so that 
$A_{L}\simeq G_{M}^{\perp}\oplus G_{K}^{\perp}(-1)$ and 
$I^{\perp} \simeq A_{L}\oplus I$. 
This implies that 
${\C}A_{L}\simeq {\C}G_{M}^{\perp}\otimes ({\C}G_{K}^{\perp})^{\vee}$ and 
\begin{equation*}
{\C}A_{L'} 
\simeq 
{\C}A_{L} \otimes {\C}G_{M} \otimes ({\C}G_{K})^{\vee} 
\simeq 
{\C}A_{L} \otimes {\rm End}({\C}G_{M}). 
\end{equation*} 
Under this isomorphism, the pullback ${\pull}$ is given by 
\begin{equation*}
{\rm id}_{{\C}A_{L}} \otimes 
\left( \sum_{\mu \in G_{M}} {\e}_{\mu}\otimes \bar{{\e}}_{\iota(\mu)} \right) 
\: \: : \: \:  
{\C}A_{L} \to {\C}A_{L'} \simeq {\C}A_{L}\otimes {\C}G_{M} \otimes ({\C}G_{K})^{\vee}.  
\end{equation*}
Note that the vector 
$\sum_{\mu \in G_{M}} {\e}_{\mu}\otimes \bar{{\e}}_{\iota(\mu)}$ 
corresponds to the identity of ${\C}G_{M}$ 
under the isomorphism 
${\C}G_{M}\otimes ({\C}G_{K})^{\vee} \simeq {\rm End}({\C}G_{M})$. 

Let $f$ be a weakly holomorphic modular form of type $\rho_L$.
By the isomorphism 
${\C}A_{L}\simeq {\C}G_{M}^{\perp}\otimes({\C}G_{K}^{\perp})^{\vee}$, 
we can write 
$f=\sum_{\nu \in G_{K}^{\perp}} f_{\nu}\otimes \bar{{\e}}_{\nu}$ 
with $f_{\nu}$ being a ${\C}G_{M}^{\perp}$-valued function. 
Then  
\begin{equation*}
f{\pull} = 
\sum_{\nu\in G_{K}^{\perp}} f_{\nu} \otimes \bar{{\e}}_{\nu} 
\otimes \left( \sum_{\mu \in G_{M}} {\e}_{\mu}\otimes \bar{{\e}}_{\iota(\mu)} \right).   
\end{equation*}
If we denote by 
$\pi \colon A_{K}\to G_{K}^{\perp}$ and 
$\pi'\colon A_{K}\to G_{K}\simeq G_{M}$ 
the natural projections, then  
\begin{equation}\label{eqn: ex1}
\langle f{\pull}, \Theta_{K} \rangle = 
\sum_{\lambda \in A_{K}} \theta_{K+\lambda} \cdot f_{\pi(\lambda)}\otimes {\e}_{\pi'(\lambda)}. 
\end{equation}
Therefore Theorem \ref{main} takes the form 
\begin{equation*}
\Psi_L(f)||_{{\DM}} = 
\Psi_M \left( \sum_{\lambda\in A_{K}} \theta_{K+\lambda} 
\cdot f_{\pi(\lambda)}\otimes {\e}_{\pi'(\lambda)} \right). 
\end{equation*}
\end{example}

\begin{example}\label{ex2}
As a special case of Example \ref{ex1}, assume that $G_{K}=A_{K}$. 
Then $G_{K}^{\perp}$ is trivial, $A_{L}\simeq G_{M}^{\perp}$, 
and $A_{M}=G_{M}^{\perp}\oplus G_{M} \simeq A_{L} \oplus A_{K}$. 
Hence ${\C}A_{M}\simeq {\C}A_{L}\otimes {\C}A_{K}$. 
Under this isomorphism, Equation \eqref{eqn: ex1} is simplified to 
\begin{equation}\label{eqn: ex2}
\langle f{\pull}, \Theta_{K} \rangle = f\otimes \Theta_{K}, 
\end{equation}
so Theorem \ref{main} takes the form    
\begin{equation*}
\Psi_L(f)||_{{\DM}} = \Psi_M(f\otimes \Theta_{K}). 
\end{equation*}
\end{example}
 
We shall look at two further special cases of Example \ref{ex2}: 
when $G_M=A_M$ and when $A_K= \{ 0 \}$. 

\begin{example}\label{ex3}
Consider the case where $L$ is unimodular. 
Nikulin shows in \cite{Ni} that  
$G_{K}=A_{K}$ and $G_{M}=A_{M}$ in this case. 
The modular form $f$ is scalar-valued because $A_{L}$ is trivial. 
Equation \eqref{eqn: ex2} is simplified to 
$\langle f{\pull}, {\ThetaK} \rangle = f\cdot {\ThetaK}$ 
where we identify $\rho_{M}\simeq \rho_{K}$ by $\iota$. 
Hence Theorem \ref{main} takes the form  
\begin{equation*}\label{eqn:quasipullback unimo}
\Psi_L(f)||_{{\DM}} = \Psi_M(f\cdot {\ThetaK}). 
\end{equation*}
This formula has been known to the experts, 
especially when $(L, f)=(II_{2,26}, 1/\Delta)$.  
See \S 16 of \cite{Bo95}, Theorem 8.5 of \cite{Yo98}, and Remark 1 in \S 6 of \cite{G-H-S07}. 
In Theorem 13.1 of \cite{Bo95}, 
Borcherds already proves that 
the weight of $\Psi_L(f)||_{{\DM}}$ equals the constant term of $f\cdot\theta_K$ 
for $M=U\oplus \langle 2d \rangle$. 
\end{example}

\begin{example}\label{ex4} 
When $K$ is unimodular, we have $G_{K}=A_{K}=\{ 0 \}$ and $L=L'$. 
The theta series $\Theta_{K}=\theta_{K}$ is scalar-valued.  
Then $\langle f{\pull}, {\ThetaK} \rangle$ is just $f\cdot \theta_K$ 
where we identify $\rho_M\simeq \rho_L$ naturally. 
Hence  
\begin{equation*}
\Psi_L(f)||_{{\DM}} = \Psi_M(f\cdot \theta_K). 
\end{equation*}
This is considered in Lemma 8.1 of \cite{Bo98}  
(see also the proof of Theorem 7.3.2 in \cite{H-MP}). 
\end{example}

\begin{example}\label{ex5}
An even lattice $L$ is called \textit{$2$-elementary} 
when $A_L\simeq({\Z}/2)^a$ for some $a\geq0$. 
Its parity $\delta$ is defined by 
$\delta=0$ if $2q(\lambda)=0\in{\QZ}$ for all $\lambda \in A_L$, 
and $\delta=1$ otherwise. 
Nikulin shows in \cite{Ni} that the isometry class of a 2-elementary lattice $L$ of signature $(2, b)$ 
is determined by the triplet $(b, a, \delta)$. 
We must have $a \equiv b$ mod $2$, 
and $\delta=0$ is possible only when $a, b$ are even. 
We write $L=L_{b,a,\delta}$ and ${\DL}=\mathcal{D}_{b,a,\delta}$ to specify these invariants. 
Then 
\begin{equation*}\label{eqn:2-elementary graph}
L_{b+r, a+r, 1} \simeq L_{b, a, \delta}\oplus \langle -2 \rangle^{\oplus r} 
\end{equation*}
for every $r>0$. 
In particular, we have a natural embedding 
\begin{equation*}
\mathcal{D}_{b,a,\delta} \simeq 
(\langle -2 \rangle^{\oplus r})^{\perp} \cap \mathcal{D}_{b+r,a+r,1} 
\hookrightarrow \mathcal{D}_{b+r,a+r,1}.  
\end{equation*} 
Moreover, if $a<a'$ with $a \equiv a' \equiv b$ mod $2$, 
there is an embedding 
$L_{b, a', \delta}\hookrightarrow L_{b, a, \delta}$ 
of finite-index. 

In \cite{Yo13}, Yoshikawa constructed a series of Borcherds products $\Psi_{b,a,\delta}$ for 
$2$-elementary lattices $L_{b,a,\delta}$ with $b\leq10$ 
which describe the analytic torsion of $K3$ surfaces with involutions. 
They are defined in Theorem 7.7 of \cite{Yo13} 
as the Borcherds lifts of the $\rho_{L_{b,a,\delta}}$-valued modular forms  
\begin{equation*}
f_{b,a,\delta} = {\rm ind}_{L_{b,a,\delta}}(\eta_{1^{-8}2^{8}4^{-8}}\theta_{\langle 2 \rangle}^{10-b}), 
\end{equation*}
where $\eta_{1^{-8}2^{8}4^{-8}}(\tau)$ is the eta product 
$\eta(\tau)^{-8}\eta(2\tau)^{8}\eta(4\tau)^{-8}$ and 
$\theta_{\langle 2 \rangle}(\tau)$ is the scalar-valued theta series of 
the lattice $\langle 2 \rangle$.  
Corollary \ref{cor:induction and theta contraction split case} tells us that we have the quasi-pullback relation 
\begin{equation*}
\Psi_{b,a,\delta} = \Psi_{b+r,a+r,1}||_{\mathcal{D}_{b,a,\delta}}, 
\end{equation*}
as was observed by Yoshikawa.  
See also p.18 of \cite{Gr10} for the case $a=b+2$.  
The modular forms $\Psi_{b,a,\delta}$ in $b<10$ 
are thus generated from the forms $\Psi_{10,a,1}$ in the line $(b, \delta)=(10, 1)$ by quasi-pullback. 

Next, for a fixed $b$ and $a<a'$ with $b \equiv a \equiv a'$ mod $2$, 
we see from Lemma \ref{lem:induction and pull push} that
\begin{equation*}
f_{b,a,\delta} = f_{b,a',\delta}{\push}, \qquad  L=L_{b,a,\delta}, \: \; L'=L_{b,a',\delta}.
\end{equation*}
Hence, as explained in Remark \ref{rmk: regularized push} below, 
$\Psi_{b,a,\delta}$ can be obtained from $\Psi_{b,a',\delta}$ 
by a sort of ``regularized average product''. 
In this sense, Yoshikawa's modular forms have two origins, 
$\Psi_{10,12,0}$ and $\Psi_{10,12,1}$. 
In \cite{Yo09}, \cite{Yo13}, he shows that 
$\Psi_{10,12,0}$ is a constant function, 
$\Psi_{10,10,0}$ is the Borcherds form $\Phi_4$ defined in \cite{Bo96}, and 
$\Psi_{10,12,1}$ is essentially the square of $\Phi_{4}$. 
\end{example}

\begin{remark}\label{rmk: regularized push}
In contrast to ${\pull}$ (Lemma \ref{lem:pull on Borcherds product}), 
the effect of the operation ${\push}$ on Borcherds products seems to be not so simple. 
It sends a Borcherds product $\Psi_{L'}$ on $\mathcal{D}_{L'}$ with 
\begin{equation*}
2{\rm div}(\Psi_{L'}) = \sum_{\mu \in A_{L'}}\sum_{n} c_{\mu}(n)Z(\mu, n)_{L'}
\end{equation*}
to a Borcherds product $\Psi_{L}$ on ${\DL}=\mathcal{D}_{L'}$ with 
\begin{eqnarray*}
2{\rm div}(\Psi_{L}) 
& = & \sum_{\mu \in I^{\perp}}\sum_{n} c_{\mu}(n)Z(p(\mu), n)_{L} \\ 
& = & \sum_{\lambda \in A_{L}}\sum_{n} \left( \sum_{\mu\in p^{-1}(\lambda)} 
c_{\mu}(n) \right) Z(\lambda, n)_{L} \\ 
& = & \sum_{\mu \in I^{\perp}}\sum_{n} \left( \sum_{\mu'\in \mu+I} 
c_{\mu'}(n) \right) Z(\mu, n)_{L'}, 
\end{eqnarray*}
where $I=L/L'\subset A_{L'}$ and $p\colon I^{\perp}\to A_L$ is the projection as before. 
This operation, a kind of regularized average product, 
may send a constant function to an interesting modular form: 
see, e.g., Examples 8.8 -- 8.12 of \cite{Yo13}. 
\end{remark}


\end{document}